\documentclass[a4paper,11pt]{amsart}

\usepackage{amssymb,latexsym,amsmath,amscd}
\usepackage{xspace}
\usepackage{enumerate}
\usepackage{graphicx}
\usepackage{vmargin}
\usepackage{todonotes}

\theoremstyle{plain}
\newtheorem{theorem}{Theorem}[section]
\newtheorem*{theorem*}{Theorem}
\newtheorem{proposition}[theorem]{Proposition}
\newtheorem{corollary}[theorem]{Corollary}

\theoremstyle{definition}

\newtheorem{remark}[theorem]{Remark}

\newcommand{\enm}[1]{\ensuremath{#1}}          

\newcommand{\cal}[1]{\mathcal{#1}}

\newcommand{\QQ}{\enm{\mathbb{Q}}}

\renewcommand{\AA}{\enm{\mathbb{A}}}
\newcommand{\PP}{\enm{\mathbb{P}}}

\newcommand{\Ll}{\enm{\cal{L}}}

\newcommand{\Oo}{\enm{\cal{O}}}

\newcommand{\Rr}{\enm{\cal{R}}}

\renewcommand{\phi}{\varphi}
\renewcommand{\theta}{\vartheta}
\renewcommand{\epsilon}{\varepsilon}

\title[Finite surjective parametrizations of surfaces]{On the existence of finite surjective parametrizations of affine surfaces}
\author{Edoardo Ballico and Claudio Fontanari}
\date{}
\thanks{This research project was partially supported by GNSAGA of INdAM and by PRIN 2017 ``Moduli Theory and Birational Classification".
}

\begin{document}

\begin{abstract}
We investigate surjective parametrizations of rational algebraic varieties, in the vein of recent work by Jorge Caravantes, J. Rafael Sendra, David Sevilla, and 
Carlos Villarino. In particular, we show how to construct plenty of examples of affine surfaces $S$ not admitting a finite surjective morphism $f: \AA^2\to S$.
\end{abstract}

\maketitle

\section{Introduction}

Effective parametrizations of rational algebraic varieties, both in the affine and in the projective case, are the object of an intense investigation 
motivated by applications to computer graphics. In particular, the research team composed by Jorge Caravantes, J. Rafael Sendra, David Sevilla, and 
Carlos Villarino has recently produced several interesting partial results with both a negative and a positive flavour. Indeed, in \cite{cssv18} they establish necessary 
conditions for the existence of a \emph{birational} surjective rational map from $\AA^2$ to an affine surface $S$, hence deducing 
plenty of examples of affine surfaces which do not admit a birational surjective parametrization from a Zariski open subset 
of $\AA^2$. On the other hand, in \cite{cssv21} they construct an algorithm that, given any parametrization of a rational projective surface $\overline{S}$, 
produces a cover of $\overline{S}$ by just three \emph{generically finite} rational charts, provided that the input parametrization satisfy certain technical assumptions 
(namely, either its base locus is empty, or its Jacobian, specialized at each base point, has rank two). 

In the present note we first introduce the notion of {\it pseudo-chart} for an integral $n$-dimensional variety $X$ as a morphism 
$f: \AA^n\to X$ such that $\dim f(\AA^n) =n$ and we provide several examples of projective varieties covered by a finite number of 
pseudo-charts with finite fibers. Next, we generalize Theorem 3.1 and Corollary 3.5 of \cite{cssv18} to \emph{finite} morphisms:

\begin{theorem}\label{a10}
Let $f: \AA^2\to S$ be a surjective finite morphism with $S$ a smooth affine rational variety. Let $\overline{S}\supset S$ be any smooth projective surface 
containing $S$. Then $\overline{S}\setminus S$ is a union of rational curves, there are at least
$\rho(\overline{S})$ of them (where $\rho(.)$ denotes the Picard number) and they generate $\mathrm{Pic}(\overline{S})\otimes \QQ$.
\end{theorem}

As a straightforward consequence we get: 

\begin{corollary}
Let $\overline{S}$ be a smooth projective rational surface containing an ample curve $C$ of positive genus. Then the affine surface $S := \overline{S} \setminus C$ does not 
admit a finite surjective morphism $f: \AA^2\to S$. 
\end{corollary}

We work over the complex field $\mathbb{C}$ (or, equivalently, over any algebraically closed field of characteristic zero).

\section{The results} 
Let $X$ be an integral quasi-projective variety. Assume that $X$ is rational. Set $n:= \dim X$. A {\it quasi-chart} of $X$ is a pair $(E,j)$, where $E$ is a proper closed subset of $\AA^n$ and $j: \AA^n\setminus E\to X$ is an open embedding.
A quasi-chart is said to be {\it principal} if either $E=\emptyset$ or $E$ is an effective divisor of $\AA^n$.
\begin{remark}
Fix a proper closed subset $E\subset \AA^n$. By a theorem of Eisenbud and Evans (\cite{ee}), $E$ is the set-theoretical intersection of at most $n$ divisors. Thus any quasi-chart $(E,j)$  is the union of at most $n$ principal quasi-charts. \end{remark}

A {\it pseudo-chart} of $X$ is a morphism $f: \AA^n\to X$ such that $\dim f(\AA^n) =n$. A pseudo-chart $f: \AA^n\to X$ is called a {\it finite pseudo-chart} (resp. {\it pseudo-chart with finite fibers}) if $f$ is finite (resp. $f$ has finite fibers). To have a finite pseudo-chart is a very strong restriction on $X$ if its image is smooth and in all cases $X$ must be affine. We recall that by Hironaka's desingularization every smooth quasi-projective variety is a Zariski open subset of a smooth projective variety. Since $f$ is flat, $f_\ast (\Oo_{\AA^2})$ is an $\Oo_U$-locally free sheaf whose rank is $\deg (f)$. 

\begin{proposition}\label{a3}
Let $Y$ be an integral quasi-projective variety of dimension $m\ge 0$. There is a degree $2$ flat and surjective morphism $\AA^1\times Y\to \PP^1\times Y$.
\end{proposition}

\begin{proof}
There is a degree $2$ flat and surjective morphism $h: \AA^1\to \PP^1$. Take $f =(h,\mathrm{Id}_Y): \AA^1\times Y\to \PP^1\times Y$.
\end{proof}

\begin{proposition}\label{a4}
$\PP^1\times \PP^1$ is covered by one degree $4$ pseudo-chart with finite fibers.
\end{proposition}

\begin{proof}
By Proposition \ref{a3} there are degree $2$ surjective, flat and finite morphisms  $h_1: \AA^2 \to \AA^1\times \PP^1$ and $h_2: \AA^1\times \PP^1\to \PP^1\times \PP^1$. Take $h_2\circ h_1$ as a surjective pseudo-chart.
\end{proof}

\begin{proposition}\label{a5}
$\PP^2$ is covered by one degree $8$ pseudo-chart with finite fibers.
\end{proposition}

\begin{proof}
There is a degree $2$ finite and flat surjective map $h: \PP^1\times\PP^1\to \PP^2$. Compose $h$ with the morphism given by Proposition \ref{a4}.
\end{proof}

\begin{proposition}\label{a6}
For any integer $n\ge 1$ the variety $(\PP^1)^n$ has a surjective, flat  and with finite fibers pseudo-chart of degree $2^{n-1}$.
\end{proposition}

\begin{proof}
Apply $n-1$ times Proposition \ref{a3}, first with $Y=(\AA^1)^{n-1}$, then with $Y =\PP^1\times (\AA^1)^{n-2}$, and so on.
\end{proof}

\begin{proposition}\label{a7}
For any integer $n\ge 1$ $\PP^n$ has a surjective, flat and with finite fibers pseudo-chart of degree $n!2^{n-1}$.
\end{proposition}

\begin{proof}
Let $X\subset \PP^r$, $r:= 2^n-1$, be the image of Segre embedding of $(\PP^1)^n$. Since $n!$ is the coefficient of $x_1\cdots x_n$ in the Taylor expansion of the polynomial $(1+x_1+\cdots +x_n)^n$, we have $\deg (X) =n!$. Thus a general linear projection gives a flat and finite morphism $f: (\PP^1)^n\to \PP^n$. Compose $f$ with any pseudo-chart given by Proposition~\ref{a6}.\end{proof}

\begin{proposition}\label{a8}
Let $X$ be a $\PP^r$-bundle over $\PP^n$. Then $X$ is covered by $n+1$  flat and with finite fibers pseudo-charts of degree $r!2^{r-1}$.
\end{proposition}

\begin{proof}
Write $X =\PP(E)$, where $E$ is a rank $r+1$ vector bundle on $\PP^n$ and call $\pi: X\to \PP^n$. Fix homogeneous coordinates $x_0,\dots,x_n$ on $\PP^n$ and write $U_i:= \{x_i\ne 0\}$. Thus $U_i\cong \AA^n$ and $\PP^n =U_0\cup \cdots \cup U_n$.
Quillen and Suslin proved that every algebraic vector bundle on $\AA^n$ is trivial (\cite{lam}). Thus $\pi^{-1}(U_i)\cong \AA^n\times \PP^r$. Apply Proposition \ref{a7} to each $\pi ^{-1}(U_i)$.
\end{proof}





\emph{Proof of Theorem \ref {a10}.}
Since $\overline{S}$ is a smooth and projective rational surface, $\mathrm{Pic}(\overline{S})$ is free abelian group of rank $\rho(\overline{S})$. Set $r:= \deg(f)$. Let $\Gamma \subset \AA^2\times S$ denote the graph of $f$ and $\Gamma '$ its closure in $\PP^2\times \overline{S}$. Let $p_1: \Gamma '\to \PP^2$ and $p_2: \Gamma '\to \overline{S}$ be the restrictions to $\Gamma '$ of the two projections of $\PP^2\times \overline{S}$. Since $\Gamma$ is the graph of a finite morphism, $\Gamma$ is closed in $\AA^2\times \overline{S}$, $p_{1|\Gamma} : \Gamma \to \AA^2$ is an isomorphism and 
$p_{2|\Gamma}: \Gamma \to \overline{S}$ is a degree $r$ finite morphism with $S$ as its image. Since $\Gamma$ is closed in $\AA^2\times \overline{S}$ and $\Gamma'$ is a projective variety, 
$p_{1}(\Gamma ' \setminus \Gamma) =\PP^2\setminus \AA^2$ and $p_{2}(\Gamma ' \setminus \Gamma) =\overline{S}\setminus S$. Since $\Gamma '$ is the graph of a degree $r$ rational map $\PP^2\dasharrow \overline{S}$, there is a finite sequence of blowing-ups of points $u: Y\to \PP^2$  and a morphism $v: Y\to \overline{S}$ such that the rational maps $f$ and $v\circ u^{-1}$ coincide 
(\cite[Th. II.7]{beau}). Since $f$ is a morphism at all points of $\AA^2$, the images of all points whose blowing-ups give $u$ are contained in $\PP^2\setminus \AA^2$ (as follows from the proof of 
\cite[Th. II.7]{beau}), i.e., $u_{|u^{-1}(\AA^2)}:u^{-1}(\AA^2)\to \AA^2$ is an isomorphism. Since $\Gamma$ is closed in $\AA^2\times \overline{S}$, we get $\overline{S}\setminus S =  v(Y\setminus u^{-1}(\AA^2))$. Thus $\overline{S}\setminus S$ is a finite union of rational curves,  say $D_1,\dots ,D_s$. To see that $D_1,\dots ,D_s$ generate $\mathrm{Pic}(\overline{S})\otimes \QQ$ (hence there are at least $\rho(\overline{S})$ of them) it is sufficient to prove that $\Ll_{|S}^{\otimes x_\Ll} \cong \Oo_S$ for any $\Ll\in \mathrm{Pic}(\overline{S})$ and some positive integer $x_\Ll$. We claim that we may take 
$x_\Ll =r$ for all $\Ll$. Indeed, fix $\Ll\in \mathrm{Pic}(\overline{S})$ and set $\Rr:= \Ll_{|S}$. Since every line bundle on $\AA^2$ is trivial, $f^\ast (\Rr)\cong \Oo_{\AA^2}$. Since $\AA^2$ and $S$ are smooth and $f$ is a finite morphism, $f$ is flat (\cite[Cor. at p. 158]{fis}). Thus $G:= f_\ast(\Oo_{\AA^2})$ is a rank $r$ vector bundle on $S$ (\cite[III, Prop. 9.2 (e) and Ex. 9.7.1]{h}). For any line bundle $R$ on $S$, the projection formula gives $f_\ast (f^\ast R) \cong G\otimes R$ for every line bundle $R$ on $S$. Since $f^\ast (\Rr)\cong \Oo_{\AA ^2}$, we get $G\otimes \Rr \cong G$. By taking determinants, we get 
$\Rr^{\otimes r}\cong \Oo_S$, thus concluding the proof.
\qed

\providecommand{\bysame}{\leavevmode\hbox to3em{\hrulefill}\thinspace}

\vspace{0.5cm}

\noindent
Edoardo Ballico \newline
Dipartimento di Matematica \newline
Universit\`a di Trento \newline
Via Sommarive 14 \newline
38123 Trento, Italy. \newline
E-mail address: edoardo.ballico@unitn.it

\vspace{0.3cm}

\noindent
Claudio Fontanari \newline
Dipartimento di Matematica \newline
Universit\`a di Trento \newline
Via Sommarive 14 \newline
38123 Trento, Italy. \newline
E-mail address: claudio.fontanari@unitn.it

\end{document}